\documentclass[11pt]{article}
\title{Some Results on $\mathcal{OL}_\infty$ Structure of Nuclear, Quasidiagonal $C^*$-algebras}
\author{Caleb Eckhardt \thanks{University of Illinois at Urbana-Champign; e-mail: ceckhard@uiuc.edu}}

\usepackage{graphicx, natbib, amssymb}
\usepackage{amsmath}
\usepackage{latexsym}
\usepackage{amsthm}

\newtheorem{theorem}{Theorem}[section]
\newtheorem{lemma}[theorem]{Lemma}
\newtheorem{corollary}[theorem]{Corollary}
\newtheorem{definition}[theorem]{Definition}
\newtheorem{proposition}[theorem]{Proposition}
\newtheorem{example}[theorem]{Example}
\newtheorem{question}[theorem]{Question}
\newtheorem{remark}[theorem]{Remark}
\numberwithin{equation}{section}

\newcommand{\V}{\Vert}
\newcommand{\vp}{\varphi}
\newcommand{\la}{\langle}
\newcommand{\ra}{\rangle}

\begin{document}
 \maketitle
\begin{abstract}
We continue the study of $\mathcal{OL}_\infty$ structure of nuclear $C^*$-algebras initiated by Junge, Ozawa and Ruan.
In particular, we prove if $\mathcal{OL}_\infty(A)<1.005,$ then
$A$ has a separating family of irreducible, stably finite representations.  As an application 
we give examples of nuclear, quasidiagonal $C^*$-algebras $A$ with $\mathcal{OL}_\infty(A)>1.$ 
\end{abstract}
\section{Introduction}
This paper continues the study of $\mathcal{OL}_\infty$-structure of nuclear $C^*$-algebras initiated by Junge, Ozawa and Ruan in \cite{Junge03}. Before describing the contents of this paper, we recall the necessary definitions and results.

Let $V$ and $W$ be $n$-dimensional operator spaces and consider the completely bounded version of Banach-Mazur distance:
\begin{equation*}
 d_{cb}(V,W)=\inf\{\V\vp\V_{cb}\V\vp^{-1}\V_{cb}:\vp:V\rightarrow W\textrm{ is a linear isomorphism}\}.
\end{equation*}
Let $A$ be a $C^*$-algebra.  For  $\lambda>1$ we say that $\mathcal{OL}_\infty(A)\leq\lambda$ if for every finite-dimensional subspace $E\subset A$, there exists a finite-dimensional $C^*$-algebra $B$ and a subspace $E\subset F\subset A$ such that $d_{cb}(F,B)\leq\lambda.$ Then define
\begin{equation*}
\mathcal{OL}_\infty(A)=\inf\{\lambda:\mathcal{OL}_\infty(A)\leq\lambda\}.
\end{equation*}
$A$ is a \emph{rigid} $\mathcal{OL}_\infty$ space if for every $\epsilon>0$ and every $x_1,...,x_n\in A$ there is a finite-dimensional $C^*$-algebra $B$ and a complete isometry $\vp:B\rightarrow A$ such that $dist(x_i,\vp(B))<\epsilon$ for $i=1,..,n.$ 

$\mathcal{OL}_\infty$ is an interesting invariant for $C^*$-algebras, particulaly when one considers the interplay between $\mathcal{OL}_\infty$ and various approximation properties of $C^*$-algebras.

It follows easily from the definition, that if $\mathcal{OL}_\infty(A)<\infty$, then there is a net of matrix algebras $(M_{n_i})$ and linear maps $\alpha_i:A\rightarrow M_{n_i}$, $\beta_i:M_{n_i}\rightarrow A$ such that $\beta_i\alpha_i$ tends to the identity on $A$ pointwise and $\sup_i\V\alpha_i\V_{cb}\V\beta_i\V_{cb}<\infty.$ Pisier showed \cite[Theorem 2.9]{Pisier96} that this implies $A$ is nuclear. Conversely,  it was shown in \cite{Junge03} if $A$ is nuclear, then $\mathcal{OL}_\infty(A)\leq 6.$ This estimate was improved in \cite{Junge04} when the authors showed that all nuclear $C^*$-algebras $A$ have $\mathcal{OL}_\infty(A)\leq3.$ So, $\mathcal{OL}_\infty$ is most useful when restricted to nuclear $C^*$-algebras.

Another important approximation property is quasidiagonality (QD).  We refer the reader to the survey article \cite{Brown04} for information on QD $C^*$-algebras. The following relationships between QD and $\mathcal{OL}_\infty$ were established in \cite{Junge03}:
\begin{equation*}
A \textrm{ is a rigid }\mathcal{OL}_\infty \textrm{ space}\quad\xrightarrow{\textup{(i)}}\quad \mathcal{OL}_\infty(A)=1\quad\xrightarrow{\textup{(ii)}}\quad A\textrm{ is nuclear \& QD}.
\end{equation*}
Blackadar and Kirchberg showed \cite[Proposition 2.5]{Blackadar01}that all 3 of the above assertions are equivalent if $A$ is either simple or both prime and antiliminal.
The main purpose of this paper is to give examples showing that the converse of (ii) does not hold in general.

In Section 2 we prove the necessary technical results used throughout the paper.  Section 3 contains our first counterexamples to (ii). Section 4 contains some results about permenance properties about $\mathcal{OL}_\infty.$  In Section 5 we prove the main result that all unital $C^*$-algebras $A$ with $\mathcal{OL}_\infty(A)<1.005$  have a separating family of irreducible, stably finite representations. This provides  a larger class of nuclear quasidiagonal $C^*$-algebras $A$ with $\mathcal{OL}_\infty(A)>1$, but also has implications for the converse of (i) which we discuss at the end of the paper.
\section{Technical lemmas}
In this section, we gather some technical lemmas needed for sections 3 and 4, and fix our notation. 

Throughout the paper, if $H$ is a Hilbert space, we let
$B(H)$ denote the space of bounded linear operators on $H.$ For $H$ $n$-dimensional we write  $\ell^2(n),$ and $M_n$ for $B(\ell^2(n)).$   We write ucp and cpc as shorthand for ``unital completely positive'' and ``completely positive contraction'' respectively. For linear maps $\vp:V\rightarrow W$ between operator spaces we write
$\vp^{(n)}$ for $id_{M_n}\otimes\vp:M_n(V)\rightarrow M_n(W),$ and $\V\vp\V_{cb}=\sup_n\V\vp^{(n)}\V.$ Furthermore if $\vp$ is injective, we write $\V\vp^{-1}\V$ for the norm of the map $\vp^{-1}:\vp(V)\rightarrow V.$
 We write $\otimes$ for the minimal tensor product of $C^*$-algebras.

The following lemma is implicit in the proof of \cite[Theorem 3.2]{Junge03}.
\begin{lemma} \label{lem:Junge03.T3.2} 
Let $0<\delta<1/\sqrt{2},$ and let $A$ be a unital $C^*$-algebra with $\mathcal{OL}_\infty(A)<1+\delta^2/2.$ Let $F\subset A$ be a finite subset.  Then there is a finite-dimensional $C^*$-algebra $B$, a linear map $\vp:B\rightarrow A$ with $||\vp||_{cb}<1+\delta^2/2$
and a ucp map $\psi:A\rightarrow B$ such that $F\subset\vp(B)$ and 
\begin{equation*}
\V\psi\vp-id_B\V_{cb}<(1+\delta^2/2)\sqrt{2(\delta^2+\delta^4/4)}. 
\end{equation*}
\end{lemma}
In \cite[Theorem 3.2]{Junge03} the authors require $\delta<1/16.$  The reason for this is to guarantee that $\psi$ is approximately multiplicative on $F.$  We will not need approximate multiplicativity in this paper, which is why we are able to relax this condition to $\delta<1/\sqrt{2}.$

We need the following slight variation of  Lemma \ref{lem:Junge03.T3.2}.


\begin{lemma} \label{lem:L2.1} Let $\lambda<(\frac{1+\sqrt{3}}{2})^{1/2}$ and A be a unital $C^*$-algebra with $\mathcal{OL}_\infty(A)<\lambda.$
 Let $F\subset A$ be a finite subset. Then there is a finite-dimensional $C^*$-algebra B, a ucp map $\psi:A\rightarrow B$ and a unital,
self-adjoint map $\vp:B\rightarrow A$ such that
\begin{enumerate}
 \item[\textup{(i)}] $\V\vp\V_{cb}<\frac{\lambda}{1-\lambda\sqrt{2(\lambda^2-1)}}$
 \item[\textup{(ii)}] $F\subset\vp(B).$
 \item[\textup{(iii)}] $\psi\vp=id_B.$
 \item[\textup{(iv)}] $\vp\psi|_F=id_F.$
\end{enumerate}
\end{lemma}
\begin{proof}  Without loss of generality suppose $F$ consists of positive elements. We apply Lemma \ref{lem:Junge03.T3.2} with $\lambda=1+\delta^2/2$ to obtain a finite-dimensional $C^*$-algebra $B$, a ucp map
$\psi:A\rightarrow B$, and a linear map $\vp:B\rightarrow A$ such that $F\subset\vp(B),$ $\V\vp\V_{cb}<\lambda$ and 
$\V\psi\vp-id_B\V_{cb}<\lambda\sqrt{2(\lambda^2-1)}<1.$

Then $\psi\vp$ is invertible in the Banach algebra of all completely bounded maps on $B.$ Let $\vp'=\vp(\psi\vp)^{-1}.$ Then
\begin{equation*}
 \V\vp'\V_{cb}\leq \V\vp\V_{cb}\V(\psi\vp)^{-1}\V_{cb}\leq \lambda\frac{1}{1-\lambda\sqrt{2(\lambda^2-1)}}.
\end{equation*}
Then $\vp'$ satisfies (i)-(iii).
Moreover, since $\psi$ is unital and $\psi\vp'=id_B,$ it follows that $\vp'$ is unital.  

Finally, let $\vp''(x)=1/2(\vp'(x)+\vp'(x^*)^*),$ for $x\in B.$ 
Then $\vp''$ is unital, self-adjoint and $\V\vp''\V_{cb}\leq\V\vp'\V_{cb}$. Since $\psi$ is positive, it follows that
$\psi\vp''=id_B.$  To see (ii), let $b\in B$ such that $\vp'(b)\in F.$  Since $F$ consists of positive elements, 
$b=\psi\vp'(b)\geq0.$ Hence, $\vp''(b)=1/2(\vp'(b)+\vp'(b)^*)=\vp'(b)\in F.$ Condition (iv) is a consequence of (ii) and (iii).
\end{proof}
\begin{lemma} \label{lem:x1}
 Let A be a unital $C^*$-algebra and let $x\in A.$ Set
\begin{equation} \label{eq:x1def}
 x_1=\left[ \begin{array}{cc} \V x\V 1 & x\\ x^* & \V x\V 1 \end{array} \right]\in M_2\otimes A.
\end{equation}
Then $\V x_1\V=2\V x\V.$
\end{lemma}
\begin{proof}
 Without loss of generality, assume that $\V x\V=1.$ Clearly $\V x_1\V\leq 2.$ For the reverse inequality, suppose that $A\subset B(H)$ unitally for some Hilbert space $H.$ By spectral theory there is a sequence of unit vectors $(\eta_k)\subset H$ such that
\begin{equation}\label{eq:approxeigen}
 \lim_{k\rightarrow\infty}\V x^*x\eta_k-\eta_k\V= 0 
\end{equation}
For each $k\in\mathbb{N}$ set
\begin{equation*}
\xi_k =\frac{1}{\sqrt{2}}\left( \begin{array}{l} x\eta_k\\ \eta_k \\ \end{array} \right) \in H\oplus H.
\end{equation*}
Then $\V\xi_k\V\leq1$ and by (\ref{eq:approxeigen}), it follows that
\begin{equation*}
 \lim_{k\rightarrow\infty} \V x_1\xi_k\V=\lim_{k\rightarrow\infty}\frac{1}{\sqrt{2}}\Big\V\left( \begin{array}{c} 2x\eta_k\\ x^*x\eta_k+\eta_k \\ \end{array} \right)\Big\V=2.
\end{equation*}
Hence, $||x_1||\geq2.$ 
\end{proof}
\begin{lemma} \label{lem:L2.1.1} Let A and B be $C^*$-algebras with $A$ unital, and $1/2<r\leq1.$ Let $\vp:A\rightarrow B$ be a cpc such that for every
$k\in\mathbb{N}$ and $a\in M_k\otimes A$ with $a\geq0$, $\V\vp^{(k)}(a)\V\geq r\V a\V.$  Then $\vp$ is injective with $\V\vp^{-1}\V_{cb}\leq (2r-1)^{-1}.$
\end{lemma}
\begin{proof} Let $n\in\mathbb{N}$ and $x\in M_n\otimes A.$  Let $x_1\in M_2\otimes(M_n \otimes A)$ be as in Lemma \ref{lem:x1}.
Then, $x_1\geq0$ and $\V x_1\V=2\V x\V.$ By assumption, we have
\begin{align*}
2r\V x\V&\leq\V\vp^{(2n)}(x_1)\V\\
&= \Big\V\left[ \begin{array}{cc} \V x\V\vp^{(n)}(1) & \vp^{(n)}(x)\\ \vp^{(n)}(x)^* & \V x\V\vp^{(n)}(1) \end{array} \right]\Big\V\\&\leq \Big\V\left[ \begin{array}{cc} \V x\V 1 & \vp^{(n)}(x)\\ \vp^{(n)}(x)^* & \V x\V 1 \end{array} \right]\Big\V\\
&\leq\V x\V+\V \vp^{(n)}(x)\V.
\end{align*}
Hence $\V\vp^{(n)}(x)\V\geq (2r-1)\V x\V,$ from which we conclude that 
\begin{equation*}
\V\vp^{-1}\V_{cb}\leq (2r-1)^{-1}.
\end{equation*}
\end{proof}

We recall Choi's Schwarz inequality for completely positive maps from   \cite{Choi74}
\begin{lemma} \label{lem:L2.3} Let $A$ and $B$ be unital $C^*$-algebras and $\psi:A\rightarrow B$ a ucp map.  Then for every $a\in A,$ we have
$\psi(a)^*\psi(a)\leq\psi(a^*a).$
\end{lemma}
\begin{lemma} \label{lem:L2.2}
Let $L_1$ and $L_2$ be Hilbert spaces and $n\in\mathbb{N}$.  Let $\vp:M_n\rightarrow B(L_1)\oplus B(L_2)$ be an injective cpc
with $\V\vp^{-1}\V_{cb}=r^{-1}<2/(\sqrt{6}-1).$  Let $\vp_i:M_n\rightarrow B(L_i)$ denote the coordinate maps of $\vp$ for $i=1,2.$
Suppose there is a $k\in\mathbb{N}$ and $a\in M_k\otimes M_n$ of norm 1 and $a\geq0$ such that 
\begin{equation*}
\V\vp_2^{(k)}(a)\V=s<(r^2+r-1)/r.
\end{equation*}
Then $\vp_1$ is injective and 
\begin{equation}
\V\vp_1^{-1}\V_{cb}\leq\Big(r-\frac{1-r^2}{1-s}\Big)^{-1}. \label{eq:L2.2.1}
\end{equation}
\end{lemma}
\begin{proof} By \cite[Theorem 2.10]{Smith83} every bounded map $\psi$ from an operator space into
$M_n$ is completely bounded with $\V\psi^{(n)}\V=\V\psi\V_{cb}.$  So, we may assume that $k=n.$
Also by \cite[Theorem 2.10]{Smith83}, to prove inequality (\ref{eq:L2.2.1}) it suffices to show that for every $x\in M_n\otimes M_n$ of norm 1, we have
\begin{equation}
\V\vp_1^{(n)}(x)\V\geq r-\frac{1-r^2}{1-s}. \label{eq:L2.2.2}
\end{equation}
By Wittstock's extension theorem  \cite[Theorem 8.2]{Paulsen02}, let 
\begin{equation*}
\widetilde{\psi}:B(L_1)\oplus B(L_2)\rightarrow M_n 
\end{equation*}
be an extension of $\vp^{-1}:\vp(M_n)\rightarrow M_n$ with $\V\widetilde{\psi}\V_{cb}=\V\vp^{-1}\V_{cb}=r^{-1}.$
 Let $\psi=r\widetilde{\psi}.$  Then $\V\psi\V_{cb}=1$ and
\begin{equation}
\psi\vp(x)=rx \quad\textrm{ for all }x\in M_n. \label{eq:L2.2.4}
\end{equation}
By the factorization theorem for completely bounded maps \cite[Theorem 8.4]{Paulsen02} there is a unital representaion $(\pi,H)$ of
\begin{equation*}
M_n\otimes(B(L_1)\oplus B(L_2))=B(L_1\otimes\ell^2(n))\oplus B(L_2\otimes\ell^2(n))
\end{equation*}
and isometries $S,T:\ell^2(n)\otimes\ell^2(n)\rightarrow H$ such that
\begin{equation}
T^*\pi(x)S=\psi^{(n)}(x)\textrm{ for every } x\in B(L_1\otimes\ell^2(n))\oplus B(L_2\otimes\ell^2(n)).\label{eq:L2.2.5}
\end{equation}
Let $q_{L_1}=\pi(1_{L_1\otimes\ell^2(n)},0)\in B(H)$ and $q_{L_2}=\pi(0,1_{L_2\otimes\ell^2(n)})\in B(H).$

We now show that the ranges of $S$ and $T$ are almost included in $q_{L_1}(H).$

Let $\xi_1\in\ell^2(n)\otimes\ell^2(n)$ be a norm 1 eigenvector for $a$ with eigenvalue 1. Let $\omega_1\in M_n\otimes M_n$
be the orthogonal projection onto $\mathbb{C}\xi_1.$ Then $\omega_1\leq a.$  Since $\vp_2$ is cp, we have $\V\vp_2^{(n)}(\omega_1)\V\leq\V \vp_2^{(n)}(a)\V=s.$ Extend $\xi_1$ to an
orthonormal basis $\xi_1,\xi_2,...,\xi_{n^2}$ for $\ell^2(n)\otimes\ell^2(n).$ For $i=1,...,n^2$ define the rank 1 operators,
\begin{equation*}
\omega_i(\eta)=\la \eta,\xi_i\ra\xi_1, \textrm{ for }\eta\in\ell^2(n)\otimes\ell^2(n).
\end{equation*}
Then 
\begin{equation}
\omega_i\omega_j^*=\delta_{i,j}\omega_1, \textrm{ for }1\leq i,j\leq n^2. \label{eq:L2.2.3}
\end{equation}

Let $\eta=\sum_{i=1}^{n^2}\alpha_i\xi_i\in\ell^2(n)\otimes\ell^2(n)$ of norm 1 and $\omega_\eta=\sum_{i=1}^{n^2}\overline{\alpha}_i\omega_i\in M_n\otimes M_n.$ By  (\ref{eq:L2.2.3}) and Lemma \ref{lem:L2.3}, it follows that
\begin{align}
 \V\vp_2^{(n)}(\omega_\eta)\V&=\V\vp_2^{(n)}(\omega_\eta)\vp_2^{(n)}(\omega_\eta)^*\V^{1/2}\notag \\
&\leq\Big(\sum_{i=1}^{n^2}|\alpha_i|^2\V\vp_2^{(n)}(\omega_1)\V\Big)^{1/2} \leq s^{1/2}.\label{align:L2.2.6}
\end{align}
Combining (\ref{eq:L2.2.4}) and (\ref{eq:L2.2.5}), we have 
\begin{equation*}
r\xi_1=r\omega_\eta(\eta)=\psi^{(n)}\circ\vp^{(n)}(\omega_\eta)\eta=T^*\pi(\vp^{(n)}(\omega_\eta))S\eta.
\end{equation*} 
Therefore, by
(\ref{align:L2.2.6})
\begin{align}
r^2&\leq \V\pi(\vp^{(n)}(\omega_\eta))S\eta\V^2 \notag \\
&= \V\pi(\vp_1^{(n)}(\omega_\eta),0)q_{L_1}S\eta\V^2+\V\pi(0,\vp_2^{(n)}(\omega_\eta))q_{L_2}S\eta\V^2 \label{align:L2.2.7}\\
&\leq \V q_{L_1}S\eta\V^2+s\V q_{L_2}S\eta\V^2. \notag 
\end{align}
Combining (\ref{align:L2.2.7}) with the fact that $S$ is an isometry, we obtain
\begin{align*}
 1&=\V q_{L_1}S\eta\V^2+\V q_{L_2}S\eta\V^2  \\
&\geq r^2-s\V q_{L_2}S\eta\V^2+\V q_{L_2}S\eta\V^2 
\end{align*}
Since $\eta\in\ell^2(n)\otimes\ell^2(n)$ was an arbitrary vector of norm 1, it follows that 
\begin{equation}
\V q_{L_2}S\V\leq \Big(\frac{1-r^2}{1-s}\Big)^{1/2} \label{eq:L2.2.8}
\end{equation}
Define $\psi^*:B(L_1)\oplus B(L_2)\rightarrow M_n$  by $\psi^*(x)=\psi(x^*)^*.$  By the complete positivity of
$\vp$ it follows that $\psi^*\vp=r\cdot id_{M_n}.$  Moreover note that 
\begin{equation*}
(\psi^*)^{(n)}(x)=S^*\pi(x)T. 
\end{equation*}
So, by replacing $\psi$ with $\psi^*$ (and hence $S$ with $T$)in the above proof we obtain
\begin{equation}
\V T^*q_{L_2}\V=\V q_{L_2}T\V\leq \Big(\frac{1-r^2}{1-s}\Big)^{1/2} \label{eq:L2.2.9}
\end{equation}
Let $x\in M_n\otimes M_n$ be arbitrary of norm 1.  By (\ref{eq:L2.2.4}), (\ref{eq:L2.2.5}), then (\ref{eq:L2.2.8}) and (\ref{eq:L2.2.9}), we have
\begin{align*}
r&=\V\psi^{(n)}\vp^{(n)}(x)\V\\
&=\V T^*\pi(\vp^{(n)}(x))S\V \\
&=\V T^*(q_{L_1}(\pi(\vp_1^{(n)}(x),0))q_{L_1}+q_{L_2}(\pi(0,\vp_2^{(n)}(x)))q_{L_2})S\V\\
&\leq  \V\vp_1^{(n)}(x)\V+\V T^*q_{L_2}(\pi(0,\vp_2^{(n)}(x)))q_{L_2}S\V\\
&\leq \V\vp_1^{(n)}(x)\V +\frac{1-r^2}{1-s}.
\end{align*}
This proves (\ref{eq:L2.2.2}) and the lemma. 
\end{proof}
We will be careful with our norm estimates throughout the paper. Thus, the technical nature of Lemma \ref{lem:L2.2}.
Colloquially it states; regardless of the value of $n\in\mathbb{N}$,  if $\vp$ is \emph{almost} a complete isometry, then either $\vp_1$ or $\vp_2$ is \emph{almost} a complete isometry. 
In particular, we have
\begin{corollary} \label{cor:C2.1}  Let $L_1,L_2,n$ and $\vp$ be as in Lemma \textup{\ref{lem:L2.2}}, but with $\V\vp^{-1}\V_{cb}=r^{-1}<125/124.$  Then either $\vp_1$ or $\vp_2$ is injective, and 
\begin{equation*}
\V\vp_i^{-1}\V_{cb}\leq (1+(r-1)^{1/3})^{-1}
\end{equation*} 
for either $i=1$ or $i=2.$
\end{corollary}
\begin{proof} If $\vp_2$ is injective with $\V\vp_2^{-1}\V_{cb}<(1+(r-1)^{1/3})^{-1},$ we are done. If not, then there is an $x\in M_n\otimes M_n$ of norm 1 such that $\V\vp^{(n)}_2(x)\V<1+(r-1)^{1/3}.$
Then Lemma \ref{lem:L2.1.1} provides an $a\in M_n\otimes M_n$ of norm 1 with $a\geq0$ such that 
\begin{equation*}
\V\vp^{(n)}_2(a)\V\leq \frac{1}{2}(1+\V\vp^{(n)}_2(x)\V)\leq\frac{1}{2}(2+(r-1)^{1/3})
\end{equation*}
We now apply Lemma \ref{lem:L2.2} with $s=\frac{1}{2}(1+\V\vp^{(n)}_2(x)\V)$ to obtain,
\begin{equation*}
\V\vp_1^{-1}\V_{cb}\leq\Big(r-\frac{1-r^2}{1-s}\Big)^{-1}\leq(1+(r-1)^{1/3})^{-1},
\end{equation*}
which holds whenever $124/125<r\leq1.$ 
\end{proof}
Finally, we recall 2 useful perturbation lemmas.
\begin{lemma}\textup{(\cite[Proposition 1.19]{Wassermann94})} \label{lem:Kirchperb}
Let A be a unital $C^*$-algebra and $N$ an injective von Neumann algebra. Let $\vp:A\rightarrow N$ be a unital self-adjoint map with
$\V\vp\V_{cb}\leq 1+\epsilon$ for some $\epsilon>0.$  Then there is a ucp map $t:A\rightarrow N$ such that 
$\V t-\vp\V_{cb}\leq\epsilon.$ 
\end{lemma}
\begin{lemma} \textup{(\cite[Lemma 2.13.2]{Pisier03})}  \label{lem:Pisierperb} Let $0<\epsilon<1$  and $X$ be an operator space.
Let $(x_i,\widehat{x}_i)_{i=1}^n$ be a biorthogonal system with $x_i\in X$ and $\widehat{x}_i\in X^*.$  Let $y_1,...,y_n\in X$ be such that
\begin{equation*}
\sum\V\widehat{x}_i\V\textrm{ }\V x_i-y_i\V<\epsilon. 
\end{equation*}
Then there is a complete isomorphism $w:X\rightarrow X$ such that $w(y_i)=x_i$ and $\V w\V_{cb}\V w^{-1}\V_{cb}\leq \frac{1+\epsilon}{1-\epsilon}.$
\end{lemma}

\section{First examples}
For $1\leq\lambda<(\frac{1+\sqrt{3}}{2})^{1/2},$ let
\begin{equation}
 f(\lambda)=\frac{\lambda}{1-\lambda\sqrt{2(\lambda-1)}},\label{eq:fdef}
\end{equation}
and consider the real polynomial,
\begin{equation}
g(y)=y(1+y)(y-1)(2-y)-2(2-y)^2+1. \label{eq:gdef}
\end{equation}
Note that $f(\lambda)\rightarrow 1$ as $\lambda\rightarrow1$, and $g(1)=-1.$
Let $\lambda'$ in the domain of $f$ be such that 
\begin{equation}
g(f(\lambda'))<0. \label{def:D3.0}
\end{equation}
A calculation shows that any $\lambda'<1.005$ satisfies $(\ref{def:D3.0})$


\begin{theorem} \label{thm:T3.3}
Let A be a unital $C^*$-algebra and let $\lambda'$ satisfy $(\ref{def:D3.0}).$ Suppose that A has a unital faithful representation
$(\pi,H_\pi)=(\rho\oplus\sigma, H_\rho\oplus H_\sigma)$, such that ker$(\sigma)\neq\{0\}.$ Furthermore suppose there is a sequence $(x_n)$ in the unit sphere of $A$ such that  $\rho(x_n)$ is an isometry for each $n$, and $\rho(x_nx_n^*)\rightarrow0$ strongly in $B(H_\rho).$  Then $\mathcal{OL}_\infty(A)\geq\lambda'.$ 
\end{theorem}
\begin{proof} Let $a\in \textrm{ker}(\sigma)$ be positive and norm 1.
Choose $n$ large enough so $\rho(1-x_nx_n^*)\rho(a)\rho(1-x_nx_n^*)\neq0.$ 
Set $y=x_n,$ and let
\begin{equation*}
 b=\V (1-yy^*)a(1-yy^*)\V^{-1}(1-yy^*)a(1-yy^*).
\end{equation*}
Then  $\sigma(b)=0,$ hence $1=\V b\V=\V\rho(b)\V.$ Since $\rho(1-yy^*)$ is a projection, it follows that $\rho(b)\leq\rho(1-yy^*),$ hence
\begin{equation}
\pi(b)\leq \pi(1-yy^*). \label{eq:T3.1.1}
\end{equation}
Suppose that $\mathcal{OL}_{\infty}(A)<\lambda',$ and obtain a contradiction.   Let $F=\{b,y,y^*\}.$ Let $f$ and $g$ be as in (\ref{eq:fdef}) and (\ref{eq:gdef}). We apply Lemma \ref{lem:L2.1} to obtain a finite-dimensional $C^*$-algebra $B$,  a ucp map $\psi:\pi(A)\rightarrow B$ and a unital, self-adjoint map $\vp:B\rightarrow \pi(A)$ such that
\begin{equation}
  \V\vp\V_{cb}<f(\lambda'),\quad  \psi\vp=id_B, \quad\textrm{ and } \quad \vp\psi|_{\pi(F)}=id_{\pi(F)} \label{eq:T3.1.2}
\end{equation}
By Lemma \ref{lem:Kirchperb}, there is a ucp map $t:B\rightarrow B(H_\rho)\oplus B(H_\sigma)$ such that
\begin{equation}
 \V t-\vp\V_{cb}<f(\lambda')-1. \label{eq:T3.1.3}
\end{equation}
Let $n\in\mathbb{N}$ and $x\in M_n\otimes B.$ Since $\psi\vp=id_B,$ it follows that $\V\vp^{(n)}(x)\V\geq\V x\V.$ Therefore,
\begin{align*}
\V t^{(n)}(x)\V&\geq \V\vp^{(n)}(x)\V-\V\vp^{(n)}(x)-t^{(n)}(x)\V\\
&\geq \V x\V-(f(\lambda')-1)\V x\V\\
&=(2-f(\lambda'))\V x\V.
\end{align*}
Hence $t$ is injective with 
\begin{equation}
 \V t^{-1}\V_{cb}\leq (2-f(\lambda'))^{-1}. \label{eq:T3.1.4}
\end{equation}
Let $q_\rho$ and $q_\sigma$ denote the orthogonal projections of $H_\rho\oplus H_\sigma$ onto $H_\rho\oplus 0$ and $0\oplus H_\sigma$
respectively.  By (\ref{eq:T3.1.2}) and (\ref{eq:T3.1.3}) we have
\begin{align}
\V q_\sigma t\psi(\pi(b))\V&\leq \V q_\sigma\vp\psi(\pi(b))\V+f(\lambda')-1 \notag\\
&=\V\sigma(b)\V + f(\lambda')-1 \notag\\
&=f(\lambda')-1.  \label{eq:T3.1.5}
\end{align}
Let $p\in B$ be a minimal central projection such that $\V p\psi(\pi(b))\V=\V\psi(\pi(b))\V.$ Then $pB\cong M_n$ for some $n\in\mathbb{N}.$ 
 Using (\ref{eq:T3.1.4}) and (\ref{eq:T3.1.5}), we apply Lemma \ref{lem:L2.2} with 
\begin{equation*}
s=\V q_\sigma t(p\psi(\pi(b)))\V\leq
\V q_\sigma t\psi(\pi(b)) \V\leq f(\lambda')-1 \quad\textrm{ and } 
\end{equation*}
\begin{equation*}
 r^{-1}=\V (t|_{pB})^{-1}\V_{cb}\leq (2-f(\lambda'))^{-1} 
\end{equation*}

to obtain,
\begin{equation}
 \V(q_\rho t|_{pB})^{-1}\V_{cb}\leq\Big(\frac{2(2-f(\lambda'))^2-1}{2-f(\lambda')}\Big)^{-1}. \label{eq:T3.1.6}
\end{equation}
 
Recall that for any finite $C^*$-algebra $C$ and any contractive $x\in C$, we have 
\begin{equation}
\V 1-xx^*\V=\V 1-x^*x\V. \label{eq:stabfin}
\end{equation}
 In particular, (\ref{eq:stabfin}) holds for any finite-dimensional $C^*$-algebra. We will use (\ref{eq:T3.1.6}) to
``isolate'' $\rho$, then the fact that $\rho(A)$ violates (\ref{eq:stabfin}) to arrive at a contradiciton.
\begin{align*}
f(\lambda')^{-1}&\leq \V\vp\V_{cb}^{-1}\leq \V\psi(\pi(b))\V \quad \quad (\textrm{ by }(\ref{eq:T3.1.2}))\\
&\\
&=\V p\psi(\pi(b))\V\\
&\\
&\leq \V p(1-\psi(\pi(y)\pi(y^*))\V \quad \quad (\textrm{ by }(\ref{eq:T3.1.1}))\\
&\\
&\leq \V p(1-\psi(\pi(y))\psi(\pi(y^*))\V \quad\textrm{ (by Lemma } \ref{lem:L2.3})\\
&\\
&= \V p(1-\psi(\pi(y^*))\psi(\pi(y))\V \quad(\textrm{ by }(\ref{eq:stabfin}))\\
&\\
&\leq \V(q_\rho t|_{pB})^{-1}\V_{cb}\V q_\rho t(p(1-\psi(\pi(y^*))\psi(\pi(y)))\V\\
&\\
&\leq \V(q_\rho t|_{pB})^{-1}\V_{cb}\V q_\rho t(1-\psi(\pi(y^*))\psi(\pi(y))\V\\
&\\
&\leq \V(q_\rho t|_{pB})^{-1}\V_{cb}\V q_\rho-q_\rho t(\psi(\pi(y^*)))t(\psi(\pi(y)))\V \quad \textrm{ (by Lemma } \ref{lem:L2.3})\\
&\\
&\leq \V(q_\rho t|_{pB})^{-1}\V_{cb}\Big(\V q_\rho-q_\rho \vp(\psi(\pi(y^*)))\vp(\psi(\pi(y)))\V+\V t-\vp\V_{cb}(1+\V\vp\V_{cb})\Big)\\
&\\
&= \V(q_\rho t|_{pB})^{-1}\V_{cb}\Big(\V\rho(1-y^*y)\V+\V t-\vp\V_{cb}(1+\V\vp\V_{cb})\Big) \quad (\textrm{ by }(\ref{eq:T3.1.2}))\\
&\\
&\leq \Big(\frac{2(2-f(\lambda'))^2-1}{2-f(\lambda')}\Big)^{-1}(f(\lambda')-1)(1+f(\lambda')).\\
\end{align*}
The last line follows because $\rho(y)$ is an isometry, by (\ref{eq:T3.1.2}), (\ref{eq:T3.1.3}) and (\ref{eq:T3.1.6}).
Hence $g(f(\lambda')>0,$ a contradiction. 
\end{proof}
In \cite{Junge03} it was asked if there were any nuclear, quasidiagonal $C^*$-algebras $A$ with $\mathcal{OL}_\infty(A)>1.$
We give some examples of such algebras. Let $\lambda'$ satisfy (\ref{def:D3.0}).

\begin{example} \label{ex:E3.4} \end{example} Let $s\in B(\ell^2)$ denote the unilateral shift.  Then, $A=C^*(s\oplus s^*)$ is nuclear and quasidiagonal.  Applying Theorem \ref{thm:T3.3} with $\rho:A\rightarrow C^*(s),$  $\sigma:A\rightarrow C^*(s^*)$ and $(x_n)=(s\oplus s^*)^n,$ we have $\mathcal{OL}_\infty(A)>\lambda'.$ 

Before the author obtained Theorem \ref{thm:T3.3}, Narutaka Ozawa outlined for me an alternate proof that $\mathcal{OL}_\infty(C^*(s\oplus s^*))>1.$ The proof was based on the observation that for any finite-dimensional $C^*$-algbera $B$ and any partial isometry $v\in B$, we have $1-vv^*$ Murray-von Neumann equivalent to $1-v^*v.$  But, if we let $(e_{ij})$ denote matrix units for $B(\ell^2)$ and $T=s\oplus s^*,$  then $T$ is a partial isometry and $1-T^*T=0\oplus e_{11}$ and $1-TT^*=e_{11}\oplus 0.$
So, $1-T^*T$ and $1-TT^*$ are not Murray-von Neumann equivalent in $C^*(s\oplus s^*)''=B(\ell^2)\oplus B(\ell^2).$  One can use these facts and arguments similar to Lemmas \ref{lem:L2.1} and \ref{lem:Kirchperb} to show that $\mathcal{OL}_\infty(C^*(s\oplus s^*)>1.$

\begin{example} \label{ex:E3.5} (\textup{\cite[Example IX.11.2]{Davidson96})} \end{example}
Let $D_1$ and $D_2$ be commuting diagonal operators with joint essential spectrum $\mathbb{RP}^2$, the real projective plane. Let $s$ be as in Example \textup{\ref{ex:E3.4}}.  Set 
\begin{equation*}
A=C^*(s\oplus N_1, 0\oplus N_2).
\end{equation*}
Then, A is easily seen to be an extension of nuclear $C^*$-algebras and hence is nuclear. As is shown in \cite{Davidson96},  $A$ is quasidiagonal.  Applying  Theorem \textup{\ref{thm:T3.3}}, with $\rho:A\rightarrow C^*(s)$, $\sigma:A\rightarrow C^*(N_1,N_2)$ and $(x_n)=(s\oplus N_1)^n$, we have  $\mathcal{OL}_\infty(A)>\lambda'.$
\section{Permenance Properties}
We now investigate a couple permenance properties of $\mathcal{OL}_\infty.$ 

Let $B\subset A$ be nuclear $C^*$-algebras with $\mathcal{OL}_\infty(A)=1.$  In general, we do not have $\mathcal{OL}_\infty(B)=1.$
Indeed let $B=C^*(s\oplus s^*)$ from Example \ref{ex:E3.4}.  It is easy to see that $s\oplus s^*$ is a compact perturbation of a unitary operator $u\in B(\ell^2\oplus\ell^2).$  Let $A=C^*(u)+K(\ell^2\oplus \ell^2).$  Then $A$ is nuclear and inner quasidiagonal \cite[Defnition 2.2]{Blackadar01}.
By \cite[Theorem 4.5]{Blackadar01}, $A$ is a strong NF algebra, which is a rigid $\mathcal{OL}_\infty$-space by \cite[Theorem 6.1.1]{Blackadar97}.  Hence $\mathcal{OL}_\infty(A)=1,$ but $\mathcal{OL}_\infty(B)>1.$

In contrast to this situation, if $B$ is an ideal we have the following:
\begin{theorem} \label{thm:ideals} Let $A$ be a unital $C^*$-algebra and $J$ an ideal of $A.$  If $\mathcal{OL}_\infty(A)=1,$ then $\mathcal{OL}_\infty(J)=1.$
\end{theorem}
\begin{proof} Let $\epsilon>0$ and $E\subset J$ a finite dimensional subspace.  Without loss of generality suppose $E$ has a basis of positive elements $x_1,...,x_n\in E$ with $\V x_i\V=1$ for each $i=1,...,n.$ Let $\widehat{x}_1,...,\widehat{x}_n\in J^*$ such that
$\la x_i,\widehat{x}_j\ra=\delta_{i,j}.$ Set $M=\sum \V\widehat{x}_i\V.$ 

Define 
\begin{equation*}
\delta_1(\delta)=(1-\delta)-(1-\sqrt{\delta})^{-1}(1-(1-\delta)^2)) \quad\textrm{ for }\quad 0\leq\delta<1.
\end{equation*}
Note that 
$\delta_1(\delta)\rightarrow1$ as $\delta\rightarrow 0.$

Choose $\delta>0$ small enough so $2\sqrt{\delta}\leq \epsilon/M$ and $(2\delta_1(\delta)-1)^{-1}\leq 1+\epsilon.$

By Lemma \ref{lem:L2.1} we obtain a finite-dimensional $C^*$-algebra $B=\oplus_{i=1}^N M_{n_i}$, a ucp map $\psi:A\rightarrow B$ and a unital, self-adjoint map $\vp:B\rightarrow A$ with $\V\vp\V_{cb}\leq 1+\delta$ such that $\vp\psi|_E=id_E$ and $\psi\vp=id_B.$ 

We will construct a finite-dimensional subspace $E\subset\widetilde{F}\subset J^{**}$ such that $\widetilde{F}$ is almost completely isometric to a finite-dimensional $C^*$-algebra and then apply a key theorem from \cite{Junge03} to obtain a subspace $E\subset F\subset J$ such that $F$ is almost completely isometric to a finite-dimensional $C^*$-algebra.

Since $A^{**}$ is injective, Lemma \ref{lem:Kirchperb} provides a ucp map $t:B\rightarrow A^{**}$ such that
$\V t-\vp\V_{cb}\leq\delta.$ Then $t$ is injective and $\V t^{-1}\V_{cb}\leq (1-\delta)^{-1}.$

Let $p\in A^{**}$ be the central projection such that $pA^{**}=J^{**}.$ Let $t_1:B\rightarrow J^{**}$ be defined by $t_1(x)=pt(x)$ and 
$t_2:B\rightarrow A^{**}$ by $t_2(x)=(1-p)t(x).$

Returning to the $C^*$-algebra $B$, let $q_1,...,q_N\in B$ be the minimal central projections such that $q_iB\cong M_{n_i}.$ Let
\begin{equation}
\mathcal{I}=\{1\leq i\leq N: \sup_{1\leq j\leq n}\V q_i\psi(x_j)\V\leq\sqrt{\delta}\} \label{eq:Idef} 
\end{equation}
Set $q=\sum_{i\not\in\mathcal{I}} q_i$ and $C=qB.$  

We now  show that $t_1:C\rightarrow J^{**}$ is injective with $\V t_1|_C^{-1}\V_{cb}\leq (2\delta_1(\delta)-1)^{-1}.$
We first show $t_1$ restricted to each summand of $C$ is almost a complete isometry.

To this end, let  $\mathcal{I}^c=\{1,...,N\}\setminus \mathcal{I}$ and $j\in \mathcal{I}^c.$  Then there is an $x_i$ such that $\V q_j\psi(x_i)\V>\sqrt{\delta}.$  Recall that $\vp\psi(x_i)=x_i$ and since $x_i\in J,$ $(1-p)x_i=0.$  Since $\psi$ is ucp, $\V q_j\psi(x_i)\V^{-1}q_j\psi(x_i)\in q_jB\cong M_{n_j}$ is norm 1 and positive. Hence,
\begin{align*}
 \V t_2(\V q_j\psi(x_i)\V^{-1}q_j\psi(x_i))\V&\leq \V q_j\psi(x_i)\V^{-1}\V t_2(\psi(x_i))\V\\
&=\V q_j\psi(x_i)\V^{-1}\V(1-p)t\psi(x_i)\V\\
&\leq \V q_j\psi(x_i)\V^{-1}\Big(\V(1-p)\vp\psi(x_i)\V+\delta\Big)\\
&\leq \sqrt{\delta}.
\end{align*}
We apply Lemma \ref{lem:L2.2} to $t_1:q_jB\rightarrow J^{**}$ with 
\begin{equation*}
s=\V t_2(\V q_j\psi(x_i)\V^{-1}q_j\psi(x_i))\V\leq\sqrt{\delta}\quad \textrm{and} 
\end{equation*}
\begin{equation*}
r^{-1}=\V t|_{q_jB}^{-1}\V_{cb}\leq\V t|_B^{-1}\V_{cb}\leq (1-\delta)^{-1}
\end{equation*}
to obtain 
\begin{equation}
\V t_1|_{q_jB}^{-1}\V_{cb}\leq \delta_1(\delta)^{-1} \quad \textrm{ for all }\quad j\in\mathcal{I}^c. \label{eq:t_1inj}
\end{equation}
Now,  let $k\in\mathbb{N}$ be arbitrary and $a=\sum_{j\not\in\mathcal{I}}(1_k\otimes q_jk)a\in M_k\otimes C$ be positive. 
Since $t_1$ is completely positive, by (\ref{eq:t_1inj}),
\begin{equation*}
\V t_1^{(k)}(a)\V \geq \sup_{j\not\in\mathcal{I}}\V t_1^{(k)}((q_j\otimes 1_k)a)\V\geq \delta_1(\delta)\V a\V. \label{eq:t_1injonC}
\end{equation*}
By Lemma \ref{lem:L2.1.1}, $t_1:C\rightarrow J^{**}$ is injective with $\V t_1^{-1}\V_{cb}\leq (2\delta_1(\delta)-1)^{-1}.$

$t_1(C)$ does not necessarily contain $x_1,...,x_n.$ We fix this with a perturbation.  Since $px_i=x_i=\vp\psi(x_i)$ for $i=1,...,n$, it follows from (\ref{eq:Idef}) that
\begin{align*}
\V x_i-t_1(\psi(x_i)q)\V&\leq \V x_i-t_1\psi(x_i)\V+\sqrt{\delta}\\
&= \V x_i-pt\psi(x_i)\V+\sqrt{\delta}\\
&\leq \V x_i-p\vp\psi(x_i)\V+\delta+\sqrt{\delta}\\
&=\delta+\sqrt{\delta}.
\end{align*}
Set $y_i=t_1(\psi(x_i)q)\in J^{**}$ for $i=1,...,n.$  Then,
\begin{equation*}
\sum_{i=1}^n \V \widehat{x}_i\V\textrm{ }\V x_i-y_i\V\leq M(2\sqrt{\delta})\leq\epsilon.
\end{equation*}

By Lemma \ref{lem:Pisierperb} there is a complete isomorphism $w:J^{**}\rightarrow J^{**}$ such that $w(y_i)=x_i$ for $i=1,...,n$ and
$\V w\V_{cb}\textrm{ }\V w^{-1}\V_{cb}\leq (1+\epsilon)/(1-\epsilon).$

Let $\widetilde{F}=wt_1(C)\subset J^{**}.$  Then $E\subset \widetilde{F}$ and 
\begin{equation*}
d_{cb}(\widetilde{F},C)\leq \frac{(1+\epsilon)}{(1-\epsilon)}(2\delta_1(\delta)-1)^{-1}< \frac{(1+\epsilon)^2}{1-\epsilon}.
\end{equation*}
By \cite[Theorem 4.3]{Junge03} there is a subspace $F\subset J$ such that $E\subset F$ and $d_{cb}(F,C)<(1+\epsilon)^2(1-\epsilon)^{-1}.$  

Since $\epsilon>0$ was arbitrary, it follows that $\mathcal{OL}_\infty(J)=1.$ 
\end{proof}
\begin{remark} It is not known if Theorem \ref{thm:ideals} holds in general, i.e. if $J$ is an ideal of $A$ do we always have $\mathcal{OL}_\infty(J)\leq\mathcal{OL}_\infty(A)$? 
\end{remark}
Finally, we need the following Proposition for Section \ref{sec:Irreducible Representations}. 
For $C^*$-algebras $A$ and $B$, let $A\odot B$ denote the algebraic tensor product of $A$ and $B.$
\begin{proposition} \label{prop:tpOL} Let $A_1$ and $A_2$ be nuclear $C^*$-algebras.  Then 
\begin{equation*}
 \mathcal{OL}_\infty(A_1\otimes A_2)\leq \mathcal{OL}_\infty(A_1)\mathcal{OL}_\infty(A_2).
\end{equation*} 
\end{proposition}
\begin{proof} Let $E\subset A_1\odot A_2$ be a finite dimensional subspace and $\epsilon>0.$ For $i=1,2$  choose finite dimensional subspaces  $F_i\subset A_i$, finite-dimensional $C^*$-algebras $B_i$ and linear isomorphisms  $\vp_i:F_i\rightarrow B_i,$
such that 
\begin{equation*}
\V\vp_i\V_{cb}\V\vp_i^{-1}\V_{cb}\leq\mathcal{OL}_\infty(A_i)+\epsilon \quad \textrm{ and }\quad E\subset F_1\odot F_2.
\end{equation*}
Let $\otimes_{min}$ denote the minimal operator space tensor product. Recall that for $C^*$-algebras the minimal operator space tensor product coincides with the minimal $C^*$-tensor product (see \cite[Page 228]{Pisier03}).  Furthermore by \cite[2.1.3]{Pisier03},
\begin{equation*}
\V \vp_1\otimes\vp_2: F_1\otimes_{min}F_2\rightarrow B_1\otimes B_2\V_{cb}\leq  \V\vp_1\V_{cb}\V\vp_2\V_{cb}
\end{equation*}
We have a similar inequality for $(\vp_1\otimes\vp_2)^{-1}=\vp_1^{-1}\otimes\vp_2^{-1}.$
Since $A_1\odot A_2$ is dense in $A_1\otimes A_2,$ it follows that 
\begin{equation*}
 \mathcal{OL}_\infty(A_1\otimes A_2)\leq \inf_{\epsilon>0}(\mathcal{OL}_\infty(A_1)+\epsilon)(\mathcal{OL}_\infty(A_2)+\epsilon)=\mathcal{OL}_\infty(A_1)\mathcal{OL}_\infty(A_1).
\end{equation*}
\end{proof}
\section{Irreducible Representations and $\mathcal{OL}_\infty$} \label{sec:Irreducible Representations}
This section contains the main theorem (Theorem \ref{thm:sepfam}).  We first recall the necessary definitions and prove some preliminary lemmas.
\begin{definition}
Let A be a unital $C^*$-algebra.  Recall that $x\in A$ is an isometry if $x^*x=1.$  An isometry is called proper, if $xx^*\neq1.$
$A$ is called finite if it contains no proper isometries.  $A$ is called stably finite if $M_n\otimes A$ is finite for every $n\in\mathbb{N}.$  We will call a representation $\pi$ of A finite (resp. stably finite) if $A/\ker(\pi)$ is finite (resp. stably finite.) 
\end{definition}
\begin{lemma} \label{lem:unitarymult}
Let $H$ be a separable Hilbert space and $x\in B(H)$ be a proper isometry.  Then there is a unitary $u\in B(H)$ such that
$(ux)^n(ux)^{*n}\rightarrow0$ strongly.
\end{lemma}
\begin{proof}  It is well-known (see \cite[Theorem V.2.1]{Davidson96}) that there is a closed subspace $K\subset H$
such that relative to the decomposition $H=K\oplus K^\perp,$ we have $x=s\oplus w$ where $s$ is unitarily equivalent to $s_\alpha$, the unilateral shift of order $\alpha$ (for some $\alpha=1,2,...,\infty$), and $w$ is a unitary in $B(K^\perp).$  In particular $s^ns^{*n}\rightarrow0$ strongly in $B(K).$

Wihout loss of generality,  assume that $w=id_{K^\perp}.$

Suppose first that $K^\perp$ is infinite-dimensional. Since $x|_K$ is a proper isometry, $K$ is also infinite-dimensional. Since $H$ is separable, $K\cong K^\perp.$ Under this identification and relative to the decomposition $H=K\oplus K$, let
\begin{equation*}
u=\left[ \begin{array}{cc} 0 & 1\\ 1& 0\\ \end{array} \right]\in B(H).
\end{equation*}
Then for $n\in\mathbb{N}$ we have,
\begin{equation*}
(ux)^{2n}(ux)^{*2n}=\left[ \begin{array}{cc} s^ns^{*n}&0\\ 0& s^ns^{*n}\\ \end{array} \right]
\end{equation*}
and
\begin{equation*}
 (ux)^{2n+1}(ux)^{*(2n+1)}=\left[ \begin{array}{cc} s^ns^{*n}&0\\ 0& s^{(n+1)}s^{*(n+1)}\\ \end{array} \right].
\end{equation*}
Hence, $(ux)^n(ux)^{*n}\rightarrow0$ strongly.

Suppose now that $\textrm{dim}(K^\perp)=n<\infty.$  Let $\{f_1,...,f_n\}$ be an orthonormal basis for $K^\perp.$ Since $s$ is unitarily equivalent to a shift, let $e_1,...,e_n\in K$ be an orthonormal set such that 
\begin{equation*}
se_i=e_{i+1}\quad \textrm{ for }\quad i=1,...,n-1 \textrm{ and }e_1\perp x(H).
\end{equation*}
 Define $u\in B(H)$ by $u(e_i)=f_i$ and $u(f_i)=e_i$ for $i=1,...,n$ and $u(\eta)=\eta$ for $\eta\perp\textrm{span}\{e_1,...,e_n,f_1,...,f_n\}.$ Then $u$ is unitary and
\begin{equation*}
 (ux)^{2n}(H)\perp\textrm{span}\{e_1,...,e_n,f_1,...,f_n\}.
\end{equation*}
Hence for every $k\geq 2n$ we have $(ux)^{2n+k}=x^k(ux)^{2n}.$  Therefore,
\begin{equation*}
 (ux)^{2n+k}(ux)^{*(2n+k)}\leq (s^k\oplus 0_{K^\perp})(s^k\oplus 0_{K^\perp})^*\rightarrow0\textrm{ strongly}.
\end{equation*}
\end{proof}

We recall the following definitions (see \cite[Section 4.1]{Pedersen79}).

Let A be a $C^*$-algebra. An ideal $J$ of A is called primitive if J is the kernel of some (non-zero) irreducible representation of A.  Let \textup{Prim}(A) denote the set of all primitive ideals of A.  For a subset $X\subset \textup{Prim}(A)$, and an ideal J of A let
\begin{equation*}
\textup{ker}(X)=\bigcap_{I\in X}I\quad\textrm{ and }\quad \textup{hull}(J)=\{I\in\textup{Prim}(A):J\subset I\}.
\end{equation*}
Then \textup{Prim(}A) is a topological space with closure operation $X\mapsto$ \textup{hull(ker($X$))} (see \textup{\cite[Theorem 4.1.3]{Pedersen79}}).

The following is an easy consequence of \cite[Theorem 4.1.3]{Pedersen79}.

\begin{lemma} \label{lem:denseker}
 Let A be a $C^*$-algebra and $X\subset \textup{Prim}(A).$ Then $X$ is dense if and only if $\textup{ker($X$)}=\{0\}.$
\end{lemma}
\begin{theorem} \label{thm:sepfam}
Let A be a separable unital $C^*$-algebra with $\mathcal{OL}_\infty(A)<\lambda',$ where $\lambda'$ satisfies \textup{(\ref{def:D3.0})}.  Then A has a separating family of irreducible, stably finite representations.
\end{theorem}
\begin{proof}  We first show that $A$ has a separating family of irreducible, finite representations.

We assume that $A$ does not have a separating family of irreducible, finite representations and prove that $\mathcal{OL}_\infty(A)>\lambda'.$ Let
\begin{equation*}
\mathcal{Y}=\{y\in A:(\exists J\in \textrm{Prim}(A))(y+J\in A/J\textrm{ is a proper isometry)}\}.
\end{equation*}
Then $\mathcal{Y}$ is not empty.  For each $y\in\mathcal{Y},$ let
\begin{equation}
 O(y)=\{J\in\textrm{Prim}(A):\V (1-y^*y)+J\V<1/4\textrm{ and }\V(1-yy^*)+J\V>3/4\}\label{eq:Oydef},
\end{equation}
\begin{equation*}
CO(y)=\textrm{Prim}(A)\setminus \textrm{hull(ker}(O(y))). 
\end{equation*}
We will now prove the following statement:
\begin{equation}
(\exists y\in \mathcal{Y})(CO(y)\textrm{ is not dense in Prim}(A)) \label{eq:notdense} 
\end{equation}
(If Prim$(A)$ is Hausdorff, then (\ref{eq:notdense}) is immediate by \cite[Proposition 4.4.5]{Pedersen79}. But Prim$(A)$ is not Hausdorff in general.)

Since $A$ is separable, let $(y_n)\subset\mathcal{Y}$ be a dense sequence.  

Suppose that (\ref{eq:notdense}) does not hold. Then $CO(y_n)$ is a dense, open subset of Prim$(A)$ for each $n\in\mathbb{N}.$
Since Prim$(A)$ is a Baire space, (\cite[Theorem 4.3.5]{Pedersen79})  the following set is dense in Prim$(A)$:
\begin{equation*}
 X=\bigcap_{n=1}^\infty CO(y_n).
\end{equation*}
If there is a $J\in X$ such that $A/J$ is not finite, then there is a $y\in\mathcal{Y}$ such that $y+J$ is a proper isometry.
Then there is an $n\in\mathbb{N}$ such that 
\begin{equation*}
\V y_ny_n^*-yy^*\V+\V y_n^*y_n-y^*y\V<1/8.
\end{equation*}
  But this implies that 
\begin{equation*}
J\in O(y_n)\cap X\subset \textrm{hull(ker}(O(y_n)))\cap X=\emptyset.
\end{equation*}
Hence for every $J\in X, $ $A/J$ is finite. Since $X$ is dense, ker$(X)=\{0\}$ by Lemma \ref{lem:denseker}.  Then $A$ has a separating family of irreducible finite representations, a contradiction.  This completes the proof of (\ref{eq:notdense}).

We now build representations $\rho$ and $\sigma$ that satisfy Theorem \ref{thm:T3.3}. Let $y\in\mathcal{Y}$ satisfy (\ref{eq:notdense}). For each $J\in CO(y)$ let $\sigma_J$ be an irreducible representation of $A$ such that $\textrm{ker}(\sigma_J)=J.$  Let $\sigma=\oplus_{J\in CO(y)}\sigma_J.$  Since $CO(y)$ is not dense, we have
\begin{equation}
\textrm{ker}(\sigma)=\bigcap_{J\in CO(y)}J=\textrm{ker}(CO(y))\neq\{0\}. \label{eq:sigma} 
\end{equation}
Let $\{J_i\}_{i\in I}\subset O(y)$ be an at most countable subset such that
\begin{equation}
\textrm{ker}(\{J_i\}_{i\in I})=\textrm{ker}(O(y)). \label{eq:rhoker}
\end{equation}
For $i\in I$ let $\rho_i$ be an irreducible representation of $A$ such that $\textrm{ker}(\rho_i)=J_i.$
Let $\rho=\oplus_{i\in I}\rho_i.$ By (\ref{eq:sigma}) and (\ref{eq:rhoker}) we have
\begin{equation}
\textrm{ker}(\rho\oplus\sigma)=\textrm{ker}(O(y))\cap\Big(\bigcap_{J\in CO(y)} J\Big) 
=\bigcap_{J\in \textrm{Prim}(A)}J=\{0\}. \label{eq:rho+sig}
\end{equation}
By  definition (\ref{eq:Oydef}), for every $i\in I,$ we have $\V 1-\rho_i(y^*y)\V<1/4.$  Hence $\rho_i(y)$ is left invertible and $\rho_i(y^*y)$ is invertible.

We note that $\rho_i(y)$ is not right invertible.  Indeed, if $\rho_i(y)$ is right invertible, then there   is a unitary $u\in \rho_i(A)$ such that
$\rho_i(y)=u|\rho_i(y)|.$ Then by (\ref{eq:Oydef}) we have
\begin{equation*}
3/4<\V 1-\rho_i(yy^*)\V=\V1-u\rho_i(y^*y)u^*\V=\V 1-\rho_i(y^*y)\V<1/4,
\end{equation*}
a contradiction. 

For each $i\in I$, let
\begin{equation*}
z_i=\rho_i(y)(\rho_i(y^*y))^{-1/2}.
\end{equation*}
Then $z_i^*z_i=1,$ but $z_iz_i^*\neq1$ because $\rho_i(y)$ is not right invertible. Hence, $z_i\in\rho_i(A)$ is a proper isometry for each $i\in I.$
Define the continuous function $f:\mathbb{R}^+\rightarrow\mathbb{R}^+$ by
\begin{equation*}
f(t)=\left\{ \begin{array}{ll} \frac{8}{3\sqrt{3}}t&\quad\textrm{if }0\leq t\leq 3/4\\ t^{-1/2}&\quad\textrm{if }t>3/4 \end{array}\right.
\end{equation*}
Let $\widetilde{x}=yf(y^*y)\in A.$  Since $\textrm{sp}(\rho_i(y^*y))\subset[3/4,1]$,  it follows that
$\rho_i(\widetilde{x})=z_i$ for each $i\in I.$ Let $x\in A$ be norm 1 such that $\rho(x)=\rho(\widetilde{x})$ (such a lifting is always possible, see \cite[Remark 8.6]{Wassermann94}). 

Let $H_i$ denote the Hilbert space associated with $\rho_i.$ For each $i\in I,$ Lemma \ref{lem:unitarymult} provides a unitary 
$u_i\in B(H_i)$ such that 
\begin{equation}
(u_iz_i)^n(u_iz_i)^{*n}\rightarrow0\quad\textrm{strongly in }B(H_i),\textrm{ as }n\rightarrow\infty. \label{eq:uizi}
\end{equation}
Since each $\rho_i$ has a different kernel, they are mutually inequivalent. So,  by \cite[Theorem 3.8.11]{Pedersen79}
\begin{equation*}
 \rho(A)''=\prod_{i\in I}\rho_i(A)''=\prod_{i\in I}B(H_i).
\end{equation*}
Set $u=\oplus_{i\in I}u_i.$  Since $\rho_i(x)=z_i,$ by (\ref{eq:uizi}) we have
\begin{equation*}
(u\rho(x))^n(u\rho(x))^{*n}\rightarrow0\quad\textrm{strongly in} \prod_{i\in I}B(H_i).
\end{equation*}
By Kaplansky's density theorem (\cite[Theorem II.4.11]{Takesaki02}) there is a sequence $(u_k)$ of unitaries from $\rho(A)$ such that
$u_k\rightarrow u$ in the strong* topology.  From this we obtain sequences $(k_r)$ and $(n_r)$  such that
\begin{equation}
(u_{k_r}\rho(x))^{n_r} (u_{k_r}\rho(x))^{*n_r}\rightarrow0\quad\textrm{ strongly as }r\rightarrow\infty. \label{eq:krnr}
\end{equation}
For each $r\in \mathbb{N}$ let $x_r\in A$ be norm 1 such that $\rho(x_r)=u_{k_r}.$  By (\ref{eq:sigma}) and (\ref{eq:rho+sig}) we apply Theorem \ref{thm:T3.3} with the sequence $(x_rx)$ and deduce that $\mathcal{OL}_\infty(A)>\lambda'.$

We now return to the general case. Suppose that $\mathcal{OL}_\infty(A)<\lambda'.$

 Let $H$ be a separable, infinite dimensional Hilbert space.  Let $K$ denote the compact operators on $H$ and $K^1$ be the unitization of $K.$  Since $K^1$ is an AF algebra, $\mathcal{OL}_\infty(K^1)=1.$
By Proposition \ref{prop:tpOL}, $\mathcal{OL}_\infty(A\otimes K^1)\leq\mathcal{OL}_\infty(A)<\lambda'.$

By the above proof there is a subset $X\subset \textrm{Prim}(A\otimes K^1)$ with $\textrm{ker}(X)=\{0\}$ and $(A\otimes K^1)/J$ finite for each $J\in X.$ 

By \cite[IV.3.4.23]{Blackadar06}, 
\begin{equation*}
\textrm{Prim}(A\otimes K^1)=\{J\otimes K^1+A\otimes I:J\in \textrm{Prim}(A), I=\{0\}, K\}. 
\end{equation*}

 So, without loss of generality we may assume $X=\{J_i\otimes K^1\}_{i\in \mathcal{I}}$ with $J_i\in\textrm{Prim} (A).$
Since $K^1$ is nuclear, $(A\otimes K^1)/(J_i\otimes K^1)=(A/J_i)\otimes K^1,$ so $A/J_i$ is stably finite.  Furthermore, by the nuclearity of $K^1$, we have 
\begin{equation*}
\{0\}=\cap_{i\in \mathcal{I}} (J_i\otimes K^1)=(\cap_{i\in \mathcal{I}}J_i)\otimes K^1.
\end{equation*}
 So, $\textrm{ker}(\{J_i\}_{i\in I})=\{0\}.$ 

\end{proof}
We are now in a position to give a new class of examples of nuclear, quasidiagonal $C^*$-algebras $A$ with $\mathcal{OL}_\infty(A)>1.$

\begin{example} \label{ex:E3.6}\end{example}
Let $A$ be a unital nuclear $C^*$-algbera without a separating family of irreducible stably finite representations (in particular any non-finite nuclear, $C^*$-algebra). Let $C(A)^1=(C_0(0,1]\otimes A)^1$ be the unitization of the cone of $A$.  Since $A$ is nuclear, so is $C(A)^1.$  By 
\textup{\cite[Proposition 3]{Voiculescu91}} $C(A)^1$ is quasidiagonal.  For $t\in (0,1]$, let $I_t=\{f\in C_0(0,1]:f(t)=0\}.$
 By \cite[IV.3.4.23]{Blackadar06} every non-essential primitive ideal of $C(A)^1$ is of the form
\begin{equation*}
 I_t\otimes A+ C_0(0,1]\otimes J
\end{equation*}
for some $J\in \textrm{Prim}(A)$ and $0<t\leq1.$  Furthermore, by \cite[IV.3.4.22]{Blackadar06},
\begin{equation*}
 (C_0(0,1]\otimes A)/(I_t\otimes A+C_0(0,1]\otimes J)\cong A/J.
\end{equation*}
From this we deduce that $C(A)^1$ cannot have a separating family of irreducible, stably finite representations, hence $\mathcal{OL}_\infty(C(A)^1)>\lambda'$ by Theorem \ref{thm:sepfam}.
\section{Questions and Remarks}
Recall from the Introduction:
\begin{question} \label{ques:JOR} \textup{(\cite[Question 6.1]{Junge03})} If $\mathcal{OL}_\infty(A)=1$, is A a rigid $\mathcal{OL}_\infty$ space ?
\end{question}
Blackadar and Kirchberg showed \cite[Theorem 4.5]{Blackadar01} that a $C^*$-algebra $A$ is nuclear and inner quasidiagonal if and only if $A$ is a strong NF algebra (\cite[Definition 5.2.1]{Blackadar97}).  In \cite{Junge03} it was shown that $A$ is a strong NF algebra if and only if A is a rigid $\mathcal{OL}_\infty$ space. Furthermore, by \cite[Proposition 2.4]{Blackadar01} any $C^*$-algebra with a separating family of irreducible quasidiagonal representations is inner quasidiagonal.  

Therefore if there is a $C^*$-algebra $A$ with $\mathcal{OL}_\infty(A)=1$, but which is not a rigid $\mathcal{OL}_\infty$ space, then $A$ cannot have a separating family of irreducible, quasidiagonal representations, but $A$ must have a separating family of irreducible stably finite representations by Theorem \ref{thm:sepfam}.  
Let $A\subset B(H)$ from Example \ref{ex:E3.5}.  Then $A+K(H)$ is stably finite and prime, hence has a faithful stably finite representation.  On the other hand by \cite{Brown84}, the unique irreducible representation of $A+K(H)$ is not quasidiagonal. Hence, $A+K(H)$ is a possible counterexample to Question \ref{ques:JOR}.

Finally, recall the question raised by Blackadar and Kirchberg:
\begin{question} \label{ques:BK} \textup{(\cite[Question 7.4]{Blackadar97})} Is every nuclear stably finite $C^*$-algebra quasidiagonal? 
\end{question}
There are some interesting relationships between Question \ref{ques:BK} and $\mathcal{OL}_\infty$ strucutre.
\begin{proposition} Let A be either simple or both prime and antiliminal.  If
\begin{equation*}
 1<\mathcal{OL}_\infty(A)<\Big(\frac{1+\sqrt{5}}{2}\Big)^{1/2}
\end{equation*}
then $A$ is (nuclear) stably finite, but not quasidiagonal.
\end{proposition}
\begin{proof} This is \cite[Corollary 2.6]{Blackadar01} combined with \cite[Theorem 3.4]{Junge03} 
\end{proof}
 In light of Theorem \ref{thm:sepfam}, we have the following similar relationship:
\begin{proposition}
Let A be a $C^*$-algebra such that every primitive quotient is antiliminal.  If
\begin{equation*}
 1<\mathcal{OL}_\infty(A)<1.005
\end{equation*}
 then some quotient of $A$ is (nuclear) stably finite, but not quasidiagonal.
\end{proposition}
\begin{proof}
This is \cite[Corollary 2.6]{Blackadar01} combined with Theorem \ref{thm:sepfam}. 
\end{proof}

A portion of the work for this paper was completed while the author took part in the Thematic Program on Operator Algebras at the Fields Institute in Toronto, ON in the Fall of 2007. I would like to thank Narutaka Ozawa for a helpful discussion about this work and my advisor Zhong-Jin Ruan for all his support.

\bibliographystyle{plain}
\bibliography{mybib}

\end{document}